\documentclass{amsart}  %[12pt]
\usepackage{graphicx}
\usepackage {appendix}
\usepackage{amsfonts}
\usepackage{url}
\usepackage{hyperref} 
\hypersetup{backref,pdfpagemode=FullScreen,colorlinks=true}
\usepackage{amsmath}
\usepackage{amssymb}
\usepackage{amscd}
\usepackage{color}
\usepackage{amsthm}
\usepackage{bm}
\usepackage{indentfirst}
\usepackage[hmargin=3cm,vmargin=3cm]{geometry}
\numberwithin{equation}{section}

\newtheorem{theorem}[equation]{Theorem} 
\newtheorem{proposition}[equation]{Proposition}
 
\newtheorem{lemma}[equation]{Lemma}

\newtheorem{conjecture}{Conjecture}
\newtheorem{definition}[equation]{Definition}
\theoremstyle{definition}

\newtheorem{notation}{Notation}

\theoremstyle{remark}

\newtheorem{remark}[equation]{Remark}
\newtheorem{example}{Example}
\newtheorem{question}{Question}

\DeclareMathOperator {\energy} {energy}

\DeclareMathOperator{\image}{\mathrm{image}}

\DeclareMathOperator{\lcs}{lcs}

\begin{document}
% \href{http://yashamon.github.io/web2/papers/conformalsymplectic.pdf}{Direct link to author's version}
\title{Locally conformally symplectic deformation of Gromov non-squeezing}
\author{Yasha Savelyev}
\thanks {Partially supported by PRODEP grant}
\email{yasha.savelyev@gmail.com}
\address{University of Colima, CUICBAS}
\begin{abstract} We prove one deformation theoretic extension of the Gromov non-squeezing phenomenon to $\lcs$ structures, or locally conformally symplectic structures, which suitably generalize both symplectic and contact structures. 
We also conjecture an analogue in $\lcs$ geometry of contact non-squeezing of
Eliashberg-Polterovich and discuss other related questions.
\end{abstract}
 \maketitle
 \section {Introduction}
We study here some analogues of Gromov non-squeezing
for locally conformally symplectic manifolds, which generalize both symplectic and contact manifolds. Let us recall the definition.
\begin{definition}
A \textbf{\emph{locally conformally symplectic manifold}} or \textbf{\emph{lcs manifold}} is a smooth $2n$-fold $M$, with a $\lcs$ structure: a non-degenerate 2-form $\omega$, with the property that for every $p \in M$ there is an open $U \ni p$ such that $\omega| _{U} = f _{U} \cdot \omega _{U} $, for some symplectic form $\omega _{U} $ defined on $U$ and some smooth positive function $f _{U} $ on $U$. In the case of our paper we always have $n \geq 2$, as in case $n=1$  there are other candidates for what should be an $\lcs$ structure.
\end{definition}

These structures have recently come into focus,
for example we have a fascinating recent theorem
of Apostolov-Dloussky \cite{citeApostolovStructures} that every complex surface with an odd first Betti number admits a
natural compatible $\lcs$ structure.  Without compatibility,
a more general existence result of this form is in
Eliashberg-Murphy~\cite{citeEliashbergMurphyMakingcobordisms}. 
 
A  basic invariant of a $\lcs$ structure $\omega$
is the Lee class, $$\alpha = \alpha _{\omega}  \in
H ^{1} (M, \mathbb{R}), $$  which we now briefly
describe.

The class $\alpha$ has the following differential 
form representative, called the Lee form and also 
denoted by $\alpha$ for simplicity.  If $U$ is an 
open set so that $\omega| _{U} = f _{U} \cdot 
\omega _{U}   $ for $\omega _{U}  $ symplectic, 
and $f _{U} $ a positive smooth function, then 
$\alpha = d (\ln f _{U} )$ on $U$. By a simple 
calculation this can be seen to give well-defined 1-form
$\alpha$, see also Lee \cite{citeLee}. The class $\alpha$ has the property that on the associated $\alpha$-covering space $\widetilde{M} $,  the lift $\widetilde{\omega} $ is globally conformally symplectic, that is $\widetilde{\omega} = f \cdot \omega _{0} $ with $\omega _{0} $ symplectic and $f>0$.
 By $\alpha$-covering space we mean the covering space associated to the normal subgroup $\ker  \langle \alpha,  \cdot \rangle  \subset \pi _{1} (M, x) $,
 where $\langle \alpha, \cdot \rangle: \pi _{1}(M,x) \to
 \mathbb{R}$ is the homomorphism $$[\gamma] \mapsto  \langle
 \alpha, [\gamma]  \rangle = \int _{S ^{1}} \gamma ^{*}
 \alpha. $$

It is moreover immediate that for a $\lcs$ form
$\omega$ $$d\omega= \alpha \wedge \omega,$$ for
$\alpha$ the Lee form as defined above. For some
authors, the pair $(\omega, \alpha)$ with $\alpha$
closed s.t. $d\omega= \alpha \wedge \omega$ is the
definition of a $\lcs$ structure. This has the
advantage of being interesting even in dimension
$2$, but in dimension at least $4$ the Lee form is
uniquely determined, so that there is no
difference of our definition with this second definition.

Let $\alpha$ be a closed 1-form on a smooth
manifold $M$.  The operator $$d ^{\alpha}: \Omega ^{k} (M) \to \Omega ^{k+1} (M),   $$ $$d ^{\alpha} (\eta) = d \eta - \alpha \wedge \eta$$ is called the Lichnerowicz differential. 
It satisfies $$d ^{\alpha} \circ d ^{\alpha} =0
$$ so that we have an associated chain
complex called the \textbf{\emph{Lichnerowicz chain complex}}.
The following is one basic example of an $\lcs$ manifold.
 \begin{example} [Banyaga] \label{example:banyaga} Let $(C, \lambda)
    $ be a contact $(2n+1)$-manifold where 
    $\lambda$ is a contact form: $$\forall p \in C: \lambda
		\wedge \lambda ^{2n} (p) \neq 0. $$ Take $M=C \times
		S ^{1}  $ with the 2-form $$\omega _{\lambda} = d
 ^{\alpha} 
    \lambda$$ for $\alpha: = pr _{S ^{1} } ^{*} d\theta   $, 
 $pr _{S ^{1} }: C \times S ^{1} \to S ^{1}  $ the projection,  and $\lambda$ likewise the pull-back of $\lambda$ by the projection $C \times S ^{1} \to C $. We call $(M,\omega _{\lambda} )$  as above the \textbf{\emph{lcs-fication}} of $(C,\lambda)$.
 \end{example}

\subsection{Symplectic and $\lcs$ non-squeezing}
Gromov's famous non squeezing theorem
~\cite{citeGromovPseudoholomorphiccurvesinsymplecticmanifolds.}, says
the following. Let $\omega _{st} = \sum _{i=1} ^{n} dp _{i} \wedge dq _{i}$
denote the standard symplectic form on $\mathbb{R} ^{2n} $,
$B _{R} $ the standard closed radius $R$ ball in
$\mathbb{R} ^{2n} $ centered at $0$, and $D ^{2} _{r}
\subset \mathbb{R} ^{2} $ the
standard radius $r$ disc.
Then for $R>r$,  there does not exist a symplectic embedding $$(B _{R}, \omega _{st})
\hookrightarrow (D ^{2}_{r} \times \mathbb{R} ^{2n-2}, \omega
_{st} \oplus \omega _{st}).
$$  

Gromov's non-squeezing is $C ^{0} $ persistent in the
following sense. The proof of this is subsumed by the proof
of Theorem \ref{cor:nonsqueezing} stated in Section
\ref{section:mainargument}.  

\begin{theorem}  \label{thm:Gromov} Let $R>r>0$ be given, and
let $\omega$ be the standard product symplectic form on $M=S ^{2} \times T ^{2n-2}  $,  satisfying $$ \langle \omega,
A  \rangle =  \pi r ^{2}, A=[S ^{2} ] \otimes [pt] \in
H _{2} (M) ,  $$  (for $\langle ,  \rangle $ the usual
pairing of homology and cohomology classes).
Then for any symplectic form
$\omega' $ on $M=S ^{2} \times T ^{2n-2}  $,  sufficiently $C ^{0}
$ close to $\omega $ there is
no symplectic embedding $\phi: B _{R} \hookrightarrow (M,
\omega')$, meaning that $\phi ^{*} (\omega ') = \omega
_{st}$.
\end{theorem}

On the other hand it is natural to ask if the above theorem continues to hold for general nearby
forms. Or formally this translates to: 
 \begin{question} \label{thm:nonrigidity} Let $R>r>0$ be
 given, and let $\omega $ be the product symplectic form as
 above, satisfying $ \langle \omega, A  \rangle = \pi r ^{2}
 $. For every $\epsilon > 0 $ is there a (necessarily
 non-closed by above) 2-form $\omega'$ on $S ^{2} \times
 T ^{2n-2}  $,  $C ^{0} $ or even $C ^{\infty}
 $ $\epsilon$-close to the symplectic form $\omega $ and
 such that there is an embedding $$\phi: B _{R}
 \hookrightarrow S ^{2} \times T ^{2n-2},   $$ with $\phi ^{*}\omega'=\omega _{st}  $? We likewise call such a map $\phi$ \textbf{\emph{symplectic embedding}}. 
 \end{question}
We cannot reduce this question to just applying Theorem
\ref{thm:Gromov}. This is because:
\begin{enumerate}
	\item A symplectic form on a subdomain of the form $\phi
(B _{R}) \subset M$  may not extend to
a symplectic form on $M$ (even if $M$ has a
symplectic form!).
\item When an extension to a symplectic form on $M$ does
exist, it may not be $C ^{0}$-close to a product form $\omega$
of the form above.
\end{enumerate}
This appears to be a very difficult question, my opinion is that at least in the $C ^{0} $ case the answer is yes, in part because it is difficult to imagine any obstruction, for example we no longer have Gromov-Witten theory for such a general $\omega'$.

We will we show that if $\omega'$ is $\lcs$ then the answer
to the above question is no, in the $C ^{1}$ case, under a mild additional condition.

One may think that recent work of 
M\"uller \cite{citeMuller} may be related to the 
present discussion. But there seems to be no 
obvious such relation as pull-backs by 
diffeomorphisms of nearby forms may not be nearby. 
Hence, there is no way to go from nearby embeddings that we work with to $\epsilon$-symplectic embeddings of M\"uller.

The following theorem is a more elementary precursor to Theorem \ref{cor:nonsqueezing}. 
\begin{theorem} \label{cor:nonsqueezingintro}  Let 
   $\omega$ be the standard symplectic form on $M =S
   ^{2} \times T ^{2n-2}  $ as above, s.t. for $A$ as above
   $ \langle \omega, A\rangle = \pi r ^{2} $. 
   There is a full volume open subspace $U  \subset M$, 
	 meaning that $vol _{g} (U) = vol _{g} (M)  $ with respect
	 to any Riemannian metric $g$, and with
	 $U$ diffeomorphic to $S ^{2} \times \mathbb{R} ^{2n-2} $, 
	 such that the following holds. Let $R>r$ be given. There exists an $\epsilon >0$ s.t. if
   $\{\omega _{t} \} $, $t \in [0,1]$, $\omega
   _{0} = \omega$ is a $C ^{1}$ continuous family of $\lcs$
	 forms on $M$, with $d _{C ^{1}}  (\omega _{t}, \omega _{0}) < \epsilon$ for all $t$, then  there is no symplectic embedding $$\phi: (B _{R}, \omega _{st}) \hookrightarrow U, $$ meaning an embedding $\phi$ such that $\phi ^{*} \omega _{1} = \omega _{st} $.  
\end{theorem}

\begin{remark} \label{remark_epsilon}
% TODO: {added remark epsilon}
In general we cannot give a formula for $\epsilon $ in terms
of $R,r$. However, in case of Theorem \ref{thm:Gromov} the
condition on $\epsilon $ can be deduced from the proof to be $\epsilon \cdot \pi _{} r ^{2} < \pi
_{} (R ^{2} - r	^{2}) $, (as intuitively expected)  provided we use the standard Kahler
metric on $S ^{2} \times T ^{2n-2}$ with the symplectic form
$\omega $.
\end{remark}

\begin{remark} \label{remark_U}
It is natural to ask if we can directly formulate
a version of the theorem for $U$, which is described 
explicitly in Theorem \ref{cor:nonsqueezing}. The main issue
is that an lcs form on $U$ may not have	a suitable lcs extension to $S ^{2} \times T ^{2n-2}$. The
extension is needed by us for Gromov compactness type
considerations. So that at least the theorem above, or the
Theorem \ref{cor:nonsqueezing}, does not
a priori say anything in this case. 
On the other had, if we try to work on $U$, then we can reduce
to the case of symplectic forms as any lcs form on a simply
connected space is symplectic up to a multiple by a non-zero
function. However, in this case there are other interesting
difficulties, and only in dimension 4 it is clear how to
surmount them see  ~\cite{citeSavelyevNearbyGromov}. 
\end{remark}

We shall see in Theorem \ref{cor:nonsqueezing} that $U$ can
be taken to be $M$, provided $\phi$ satisfies a certain mild
complex linearity condition on its differential, whenever it
intersects a fixed real co-dimension 2 hypersurface in $M$, of a certain
kind.  The $C ^{1}$ continuity is used to establish energy controls for certain pseudo-holomorphic
curves, as Gromov-Witten theory behaves very
differently in $\lcs$ setting. This is relaxed in
Theorem \ref{cor:nonsqueezing} to certain
$\mathcal{T}  ^{0}$ continuity, close to $C
^{0}$ continuity. Relaxing this further to $C ^{0}$
continuity would probably require substantially new ideas.

Note that Frechet smooth $\lcs$ deformations
$\{\omega _{t} \} $ of our symplectic form $\omega$, with Lee forms $\alpha _{t} $ likewise smoothly varying in $t$, are obstructed unless $\alpha _{t}$ are DeRham exact, as pointed out to me by Kevin Sackel. This can be verified by an elementary calculation by taking the $t$ derivative at $0$ of the equation: 
\begin{equation*}
d ^{\alpha} \omega _{t} = \alpha _{t} \wedge \omega _{t}.
\end{equation*}
But our families are not required to be smooth so that non-trivial $\lcs$ deformations of a symplectic form may exist.   
This motivates the question:
\begin{question} Do there exist (continuous)
   $\lcs$ deformations $ \{(\omega _{t}, \alpha
   _{t}) \}$ of the standard product symplectic
   form on $S ^{2} \times T ^{2n-2}  $, $\alpha
   _{t}$ the Lee form of $\omega _{t}$,   so that $\alpha _{t} $ are not DeRham exact?
\end{question}
\begin{remark} \label{remark:}
Another direction for the future is to consider ``lcs deformation
of Gromov non-squeezing'' (in the sense of the theorem
above)  for symplectic manifolds $(M, \omega)$ (with finite
Gromov width) satisfying:
\begin{itemize}
	\item $\wedge:
H ^{1} (M, \mathbb{R} ^{} ) \otimes H ^{2} (M, \mathbb{R}
^{} ) \to H ^{3} (M, \mathbb{R} ^{} )$ is the zero map.
\item $H ^{1} (M, \mathbb{R} ^{}
) \neq 0 $.
\end{itemize}
In this case, the obstruction to non-exact lcs
deformation vanishes. Of course the above assumption is very
strong and weaker assumptions would suffice. We do not carry
out this idea here, as finding appropriate examples is an
interesting a problem by itself, and we would also need new Gromov-Witten
theory computations which might be outside our scope. However,
the essential strategy should be the same.
\end{remark}
\subsubsection {Toward direct generalization of contact non-squeezing} \label{section:contactnonsqueezing} 

The Eliashberg-Kim-Polterovich contact non-squeezing theorem as stated 
by Fraser ~\cite{citeFraserNonsqueezing} has the following form. 
Let $C = R ^{2(n-1)}  \times S ^{1}  $, $S ^{1} =
\mathbb{R} ^{}/\mathbb{Z}  $,  be the
prequantization space of $R ^{2n-2} $, or in other
words the contact manifold with the contact form
$d\theta - \lambda$, for $\lambda =
\frac{1}{2}(ydx - xdy)$. 
Let $B _{R} $ denote the open radius $R$ ball in
$\mathbb{R} ^{2n-2} $, and $\overline{B} _{R}$ its
topological closure. 
\begin{theorem}
   [Eliashberg-Kim-Polterovich~\cite{citeEKPcontactnonsqueezing},
   Fraser~\cite{citeFraserNonsqueezing},
   Chiu~\cite{citeChiuNonsqueezing}] 
   \label{thm:contactnonsqueezing}
For $R \geq 1$ there is no
contactomorphism  $\phi: C \to C$, isotopic to the identity, so that  $\phi
(\overline{B} _{R} \times S ^{1}) \subset  B _{R}
\times S ^{1} $.
\end{theorem}

A Hamiltonian conformal symplectomorphism of an
$\lcs$ manifold $(M,\omega)$, which we just
abbreviate by the short name:
\textbf{\emph{Hamiltonian lcs map}}, is a
$\lcs$ diffeomorphism $\phi _{H} $ generated
analogously to the symplectic case  by a smooth
function $H: M \times [0,1] \to \mathbb{R} $.
Specifically, we define the time dependent vector
field $X _{t} $  by:
\begin{equation*}
\omega (X _{t}, \cdot) = d ^{\alpha} H _{t},
\end{equation*}
for $\alpha$ the Lee form, and then taking $\phi
_{H} $ to be the time 1 flow map of $\{X _{t} \}$.
For example, let  $(C \times S ^{1}, \omega
_{\lambda}) $  be the $\lcs$-fication of a contact
manifold $(C,
\lambda)$ as above.

If $\forall t: H _{t} =-1$ then $d ^{\alpha} (H
_{t}) = \alpha$ and clearly $$X _{t} = (R
^{\lambda} \oplus 0),$$  as a section of $TC
\oplus TS ^{1}$  with $R ^{\lambda} $ the
$\lambda$-Reeb vector field. The latter is the vector field
defined by:
$$ d\lambda (R ^{\lambda}, \cdot ) = 0, 
\quad \lambda (R ^{\lambda}) = 1.$$
Thus, in this case the
associated flow is naturally induced by the Reeb
flow.   More generally, given a smooth contact isotopy
$\{\phi _{t}\}$, $\phi _{t}: C \to C$
contactomorphism of a closed contact
manifold $C$, s.t. $\phi _{0} =id$,  there is a similarly induced
Hamiltonian isotopy $\{\widetilde{\phi} _{t}\} $
on the lcs-fication $C \times S ^{1}$, s.t. $\{pr
_{C} \circ \widetilde{\phi} _{t}\} = \{\phi
_{t}\}$, for $pr_C: C \times S ^{1} \to C$ the
projection. 
 This is
left as an exercise for the reader. Thus, the
following conjecture is a direct generalization of the
contact non-squeezing Theorem \ref{thm:contactnonsqueezing}.
\begin{conjecture} [see also Oh-Savelyev
   ~\cite{citeSavelyevOh}]  \label{q:contactnonsqueezing} If $R
   \geq 1$ there is no compactly supported,
   Hamiltonian lcs map $$\phi: \mathbb{R} ^{2n} \times S ^{1} \times S ^{1} \to \mathbb{R} ^{2n} \times S ^{1} \times S ^{1},  $$ so that $\phi (\overline{U} ) \subset U$, for $U := B _{R} \times S ^{1} \times S ^{1}  $ and $\overline{U} $ the topological closure. 
  \end{conjecture}

\section {Topology on the space of $\lcs$ forms
and $J$-holomorphic curves} \label{section:proofs}
Theorem \ref{cor:nonsqueezingintro}  is stated for the standard $C ^{1}$ topology on the space of differential forms. However, this can be
relaxed to use a certain natural  $C ^{0}$ style
   topology $\mathcal{T}_{0}$, specific to $\lcs$
   forms. We will now discuss this. Let $M$ be a closed smooth manifold of dimension at least 4.  The metric topology $\mathcal{T} ^{0} $ on the set $LCS (M)$ of smooth $\lcs$ $2$-forms on $M$ will be defined with respect to the following metric.  
\begin{definition} \label{def:norm} Fix a Riemannian metric $g$ on $M$. For $\omega _{1}, \omega _{2} \in LCS (M)  $ define
\begin{equation*}
d _{0}  (\omega _{1}, \omega _{2}  )= d _{C ^{0}}  (\omega _{1}, \omega _{2} ) + d _{C ^{0} } (\alpha _{1}, \alpha _{2}),
\end{equation*}
for $\alpha _{i} $ the Lee forms of $\omega _{i} $ and $d _{C ^{0} } $ the  usual $C ^{0} $ metric induced by $g$. 
In general $d _{C ^{k}}$ will denote the usual $C ^{k}$
metric, induced by $g$.

\end{definition}  
\begin{proposition}
   \label{prop:C1deformation} The metric $d _{0}$
   on $LCS (M) $ 
   is continuous with respect to the usual $C ^{1}$
   metric.     
\end{proposition}
\begin{proof}
The following argument was suggested to me by Vestislav
Apostolov. Denote by $\Lambda  (TM) $ the vector bundle
over $M$ with fiber $\Lambda  (TM) _{p}$ over $p$, the
alternating tensor algebra $\Lambda (T _{p}M) $.   Let
$\Lambda ^{2} (TM) $ denote the sub-bundle of degree $2$
elements.  Let $\Phi ^{2} (M) := \Omega  (\Lambda ^{2}
(TM))$ denote the space of $C ^{\infty}$ sections of
$\Lambda ^{2} (TM)$ with $C ^{0}$ topology.   Likewise,
$\Lambda (T ^{*}M) $ will denote the bundle whose fiber
over $p$ is  the alternating tensor algebra $\Lambda
(T_{p} ^{*}M) $.

Let $\Theta ^{2} (M) $ denote the space of non-degenerate $C
^{\infty}$  differential 2-forms on $M$  with $C ^{0}$
topology.   We first construct a continuous map: $$\phi:
\Theta ^{2} (M)  \to \Phi ^{2} (M). $$ Let $\omega$ be
a non-degenerate $2$-form,  so that for each $p \in M$ we
get an isomorphism $i_{\omega}: T _{p}M \to T ^{*} _{p}M$,
$i _{\omega} = \omega (v, \cdot) $. Let $i ^{-1} _{\omega}$
denote the inverse  of this map. Then for each $p \in M$ we
have  a bi-linear form $\omega ^{-1} _{p}  $ on $T ^{*} _{p}
(M) $ defined by $\omega ^{-1} _{p} (\eta, \mu) = \eta (i
^{-1} (\mu))$.   This is readily seen to be skew-symmetric.
Hence, determines a section $\omega ^{-1} \in \Phi ^{2} (M)
$. We then set $\phi (\omega) = \omega ^{-1}$, so that
$\phi$ is  continuous by construction.   

Now for $\omega \in LCS (M) $  define the one-form $\eta$ on
$M$ as follows. Let $v \in T _{p}M$ then $$\eta _{p}   (v)
=(d \omega) _{p} (v \wedge \phi (\omega) _{p}), $$ so that
$v \wedge \phi (\omega) _{p} \in \Lambda ^{3} (T _{p}M)
$ and  $(d \omega) _{p} \in \Lambda ^{3} (T ^{*} _{p}M)
$ identified with a functional in   $(\Lambda ^{3} (T
_{p}M)) ^{*}$. Taking a basis for $T _{p}M$ so that $\omega
_{p}$ in this basis is the standard symplectic form, it
is easily verified that $$\forall p \in M: \eta _{p} = (n-1)
\alpha _{p},$$ for $\alpha $ the Lee form satisfying $d
\omega  = \alpha  \wedge \omega$, and where $2n $ is the
dimension of $M$.   We have thus obtained a map $LCS (M) \to
\Omega (T ^{*}M) $, which takes an $\lcs$ form and produces
its Lee form, and which is continuous with respect to the $C
^{1}$ topology on $LCS (M) $ and the $C ^{0}$ topology on
the space of $1$-forms.   Clearly the result follows.    
\end{proof}

The following characterization of convergence will be helpful.
\begin{lemma} \label{lemma:lcsconvergence}
Let $M$ be as above and let $\{\omega _{k}\} \subset LCS (M)  $ be a sequence $\mathcal{T} ^{0} $ converging to a symplectic form $\omega$. Denote by $\{\widetilde{\omega} _{k}  \}$ the lift sequence on the universal cover $\widetilde{M} $. Then there is a sequence $\{ \widetilde{\omega} _{k} ^{symp}\}  $ of symplectic forms on $\widetilde{M} $, and a sequence  $\{f _{k}\}  $ of positive functions pointwise converging to $1$, such that $ \widetilde{\omega} _{k} =  f _{k} \widetilde{\omega} _{k} ^{symp}   $.
\end{lemma}
\begin{proof} We may assume that $M$ is connected. Let $\alpha _{k} $ be the Lee form of $\omega _{k} $, and $g _{k}$
functions on $\widetilde{M} $ defined by $g _{k}
   ([p]) = \int _{[0,1]} p ^{*} \alpha _{k} $,
where the universal cover $\widetilde{M} $ is
   understood as the set of equivalence classes of paths $p$ starting at a fixed $x _{0} \in M $,
  with a pair $p _{1}, p _{2}  $ equivalent if $p _{1} (1) = p _{2} (1)  $ and $p _{2} ^{-1} \cdot p _{1}$ is null-homotopic, where $\cdot$ is the path concatenation.

Then we get:
\begin{equation*}
   d \widetilde{\omega} _{k} = dg _{k} \wedge \widetilde{\omega} _{k},  
\end{equation*}
so that if we set $f _{k}:= e ^{g _{k}}  $ then
\begin{equation*}
   d (f ^{-1} _{k} \widetilde{\omega} _{k}) =0.
\end{equation*}
Since by assumption $|\alpha _{k}| _{C ^{0}}  \to 0$, then pointwise $g _{k} \to 0 $ and pointwise $f _{k} \to 1$, so that if we set $$\widetilde{\omega} ^{symp}  _{k}:= f ^{-1} _{k} \widetilde{\omega} _{k}$$ then we are done.
\end{proof}

\begin{definition} We say that a pair $(\omega, J)$ of an
$\lcs$ form $\omega$ on $M$ and an almost complex structure
$J$ on $M$ are \textbf{\emph{compatible}} if $\omega (\cdot,
J \cdot)$ defines a $J$-invariant inner product on $M$. For
other basic notions of $J$-holomorphic curves we refer the
reader to
\cite{citeMcDuffSalamonJholomorphiccurvesandsymplectictopology}.
\end{definition}
\begin{proposition} \label{thm:noSkycatastrophe} %    , satisfying: 
Let $M$ be as above, $A \in H _{2} (M) $ fixed, and $\{\omega _{t} \}$, $t \in [0,1]$, a  $\mathcal{T} ^{0} $ continuous  family of $\lcs$ forms on $M$. Let $\{J _{t} \}$ be a Frechet smooth family of almost complex structures, with $J _{t} $ compatible with $\omega _{t} $ for each $t$.
Let $D \subset \widetilde{M} $, with $\pi: \widetilde{M} \to
M$ the universal cover of $M$, be a fundamental domain, and
$K:= \overline {D}$  its topological closure. Suppose that
for each $t$, and for every  $x \in \partial K$ (the
topological boundary) there is a $\widetilde{J} _{t}
$-holomorphic hyperplane (real codimension 2 submanifold) $H _{x} $ through $x$, with $H _{x} \subset K $,  such that $\pi (H _{x}) \subset M$  is a closed submanifold and  such that $A \cdot \pi_*([H _{x}]) \leq 0$.
Define:
\begin{equation*}
 e _{t} (u):= \int _{\mathbb{CP} ^{1} } {u} ^{*} {\omega} _{t}.
\end{equation*}
 Then
\begin{equation*}
\sup _{u,t} e _{t} (u)  < \infty,
\end{equation*}
where the supremum is over all pairs $(u,t)$, $u:
   \mathbb{CP} ^{1} \to M$ is  $J _{t}
   $-holomorphic and in class $A$.
\end{proposition}
\begin{proof} 
\begin{lemma} \label{lemma:K}
Let $M$, $A$ be as above, let $D \subset \widetilde{M} $, with $\pi: \widetilde{M} \to  M$ the universal cover of $M$, be a fundamental domain, and $K:= \overline {D}$  its topological closure. Let $(\omega,J)$ be a compatible $\lcs$ pair on $M$ such that for every $x \in \partial K$ there is a $\widetilde{J}$-holomorphic (real codimension 2) hyperplane $H _{x} \subset K \subset \widetilde{M}  $ through $x$, such that $\pi (H _{x}) \subset M$ \text{ is a closed submanifold and } such that $A \cdot [\pi (H _{x} )] \leq 0$.
Then any genus $0$, $J$-holomorphic class $A$ curve $u$ in $M$ has a lift $\widetilde{u} $ with image in $K$.
\end{lemma}
\begin{proof}
For $u$ as in the statement, let $\widetilde{u} $ be a lift
intersecting the fundamental domain $D$, (as in the
statement of main theorem). Suppose that $\widetilde{u}
$ intersects $\partial K$, otherwise we already have $\image
\widetilde{u} \subset K ^{\circ} $, for $K ^{\circ} $ the
interior, since $\image \widetilde{u}$ is connected (and by
elementary topology).  Then $\widetilde{u} $ intersects $H
_{x} $ as in the statement, for some $x$. So $u$ is
a $J$-holomorphic map intersecting the closed submanifold
$\pi (H _{x} )$ with $A \cdot [\pi (H _{x} )] \leq 0$. By
positivity of intersections \cite [Section
2.6]{citeMcDuffSalamonJholomorphiccurvesandsymplectictopology}, which in this case is just a simple exercise, $\image u \subset \pi (H _{x})$, and so $\image \widetilde{u} \subset H _{x}  $, and so $\image \widetilde{u} \subset \partial K$.
\end{proof}
  
Now returning to the proof of the proposition, let $u: \mathbb{CP} ^{1} \to M$ be a $J _{t} $-holomorphic class $A$ curve. By the lemma above $u$ has a lift $\widetilde{u}$ contained in the compact $K \subset \widetilde{M} $.  Then 
% for every $\epsilon>0$ there is a $N$ so that for $k>N$ 
we have: 
\begin{equation*}
  e _{t} (u)= \int _{ \mathbb{CP} ^{1} } \widetilde{u} ^{*} \widetilde{\omega} _{t}  \leq 
 C _{t} \langle \widetilde{\omega } _{t}  ^{symp}, A  \rangle,
\end{equation*}
   where $\widetilde{\omega} _{t} = f _{t}
   \widetilde{\omega}  ^{symp} _{t} $, for
   $\widetilde{\omega} ^{symp} _{t} $ symplectic
   on $\widetilde{M}$, and  $f _{t}: \widetilde{M}
   \to \mathbb{R}$ positive function constructed
   as in the proof of Lemma
   \ref{lemma:lcsconvergence}, and where $C _{t} =
   \max _{K} f _{t}   $. Since $\{\omega _{t} \}$
   is continuous in $\mathcal{T} _{0} $, we have
   that $\{f _{t}\} $, $\{\widetilde{\omega} _{t}
   ^{symp}   \}$  are $C _{0} $ continuous
   families in $t$. 
In particular $$C= \sup _{t} \max _{K} f _{t}  $$ and $$D= \sup _{t} \langle \widetilde{\omega} _{t'}  ^{symp}, A  \rangle$$ are finite.
And so 
\begin{equation*}
   \sup _{(u,t)} e _{t} (u) \leq C \cdot D,
\end{equation*}
where the supremum is over all pairs $(u,t)$, $u$
is $J _{t} $-holomorphic, class $A$, curve in $M$ as above.
\end{proof}
\section {Quick review of genus 0 Gromov-Witten
theory} Let $M$ be a compact smooth manifold with a pair 
$(\omega, J)$ for $\omega$  a non-degenerate
smooth 2-form and $J$ an almost complex structure.
We assume that $\omega (\cdot, J
\cdot) $ is a $J$-invariant inner product on $M$, and such
a $J$ is called $\omega $-compatible.
We will call the above data $(M,\omega,J) $  an
\textbf{\emph{almost symplectic triple}}.   

Let $$ \mathcal{M} _{0,n}   (J, A) = \mathcal{M}
_{0,n}  (M, J, A)$$ denote the moduli space of
isomorphism classes of class $A$, $J$-holomorphic
curves in $M$, with domain the Riemann sphere,
with $n$ marked labeled points $\{x _{1}, \ldots x
_{n}\}$.  In other words, $ \mathcal{M} _{0,n}
(J, A) $  is the set of isomorphism classes of
tuples $(u, \{x _{1}, \ldots, x _{n}\}) $,   where
$u: \mathbb{CP} ^{1} \to M$ is a $J$-holomorphic
map. Here an isomorphism between $(u _{1}, \{x 
_{1}, \ldots, x _{n} \})
$ and $(u _{2}, \{x' 
_{1}, \ldots, x' _{n} \})$  is a biholomorphism   $\phi:
\mathbb{CP} ^{1}  \to \mathbb{CP} ^{1}   $, s.t.
$\phi (x _{i}) = x _{i}' $  and s.t. $u_2 \circ \phi = u _{1} $.
Let $$ e _{\omega}:  \mathcal{M} _{0,n}   (J, A)
\to \mathbb{R} ^{}, $$ 
be the energy:
$$e _{\omega} ([u]) := e _{\omega} (u) := \int _{\mathbb{CP} ^{1}
} u ^{*} \omega, $$
where we take any representative $u$ of the class
$[u] $. 
(Note that this (up to a
factor) is the $L
^{2}$ energy of the map $u$  with respect to
appropriate inner products, see
~\cite
[Section 2.2]
{citeMcDuffSalamonJholomorphiccurvesandsymplectictopology}). 
\begin{notation}
   \label{notation:class} In what follows we
   usually neglect to distinguish classes and
representatives. As this should be clear from
context. So from now on we just write $u$.
\end{notation}

Let $ \{(M, \omega _{t}, J _{t}) \}$, $t \in
   [0,1] $,   be  a family of almost symplectic
   triples with $\{(\omega _{t}, J _{t}) \}$
   varying smoothly in $t$. We will say that
$\{(M, \omega _{t}, J _{t}) \}$ is a \textbf{\emph{smooth family
of almost symplectic triples}}. 
Given a smooth family of almost symplectic triples 
$ \{(M, \omega _{t}, J _{t}) \}$, $t \in
   [0,1] $,
we denote by 
$${\mathcal{M}} _{0,n}   (\{J _{t} \}, A)$$
the space of pairs $(u,t)$, $u \in
{\mathcal{M}} _{0,n}(J _{t}, A)$.
(Dropping the marked points from the notation.) 

The following is well known and follows by the same argument
as \cite[Theorem
5.6.6]{citeMcDuffSalamonJholomorphiccurvesandsymplectictopology}. 
\begin{theorem} \label{thm:complete} Let $(M,
   \omega,J) $ be as above. 
 Then $\mathcal{M} _{0,n} (M, J, A)$  has a pre-compactification  
   \begin{equation*}
\overline{\mathcal{M}} _{0,n}   (M, J, A), 
\end{equation*}
by Kontsevich stable maps, with respect to the natural
metrizable Gromov topology  \cite [Chapter
5.6]{citeMcDuffSalamonJholomorphiccurvesandsymplectictopology}.
Moreover given $E>0$,   the subspace
$\overline{\mathcal{M}} _{g,0}   (J,
 A) _{E} \subset \overline{\mathcal{M}}_{g,0}   (J,
 A) $ consisting of elements $u$ with $e _{\omega} (u) \leq E$ is
 compact.  In other words $e=e _{\omega}$ is a proper function on $\overline{\mathcal{M}}_{g,0}   (J,
 A)$.
Similarly, if $\{(M, \omega _{t}, J _{t}) \}$ is a
smooth family of almost symplectic triples, and we
   define $$e: \overline{\mathcal{M}}_{0,n}   (\{J
   _{t} \}, A) \to \mathbb{R} $$
   by $$e (u,t) = e _{\omega _{t}} (u),  $$ then
   $e$ is a proper function.
\end{theorem}
Thus, the most basic situation where we can talk
about Gromov-Witten ``invariants'' of $(M, J)$ is when the $\energy$ function is bounded on $\overline{\mathcal{M}} _{g,0}   (J, A)$.
In this case $ \overline{\mathcal{M}} _{g,n}   (J, A)$ is
compact, and has a virtual moduli cycle as in the original
approach of Fukaya-Ono \cite{citeFukayaOnoArnoldandGW}, or
the more algebraic approach of Pardon
\cite{citePardonAlgebraicApproach}.
So we may define, as usual, functionals called the
Gromov-Witten invariants:
\begin{equation} \label{eq:functional1}
GW _{g,n}  (A,J): H_* (\overline{M} _{g,n}) \otimes H _{*} (M ^{n} )  \to
   \mathbb{Q},
\end{equation}
where $\overline {M} _{g,n} $ denotes the
compactified moduli space of Riemann surfaces.
Of course closed symplectic manifolds with any
tame almost complex structure is one class of
examples, where these functionals are defined, as
in that case we have a priori bounds on the energy
of holomorphic curves in a fixed class.

Even when defined, these functionals will not in general be
$J$-invariant, but it is immediate, again by
Pardon ~\cite{citePardonAlgebraicApproach},
that they are invariant for a smooth family $\{J _{t} \}$, $t \in [0,1]$ such that the corresponding ``cobordism moduli space'': $\overline{\mathcal{M}} _{g,0}   (\{J _{t} \}, A),$
is compact.

In the case of the main argument ahead we can actually avoid
virtual moduli cycle theory, and base the argument on
standard theory of 
McDuff-Salamon~\cite{citeMcDuffSalamonJholomorphiccurvesandsymplectictopology},
once we establish compactness.
For given a $J$ on $M = S ^{2} \times T ^{2n-2}$ compatible
with an lcs structure we can preclude bubbling for
a sequence of $J$-holomorphic curves in the class $A= [S ^{2}] \otimes [pt] \in H _{2} (S ^{2} \times
T ^{2n-2}, \mathbb{Z}) $
using the following.
\begin{lemma} \label{lemma_bubbling} Let $(M, \omega, J) $ be
an almost symplectic triple with $\omega$ an lcs form.
Suppose further $H _{2} (M, \mathbb{Z})
= \mathbb{Z}$ and is generated by $A$ having
a representative $u: S ^{2} \to M$ satisfying: 
\begin{equation*}
\int _{S ^{2} } u ^{*} \omega >0. 
\end{equation*}
Then if $v: S ^{2} \to M$ is a non-constant $J$-holomorphic
map  $[v] = c \cdot A $ with $c>0$.
\end{lemma}
Note that the above does not hold for a general almost
symplectic manifold. Using the lemma above we see that any
$J$-holomorphic stable map into $M$, with non-homologous
components, cannot be in class $A$, unless it has just one
component. 
\begin{proof} [Proof]
Let $\widetilde{M} $ denote the universal cover and let
$\widetilde{\omega } $, $\widetilde{J} $ be the lift of
$\omega $ and $J$ respectively. Then
$\omega $ is globally conformally symplectic as the
obstruction Lee class $\alpha \in H ^{1} (M, \mathbb{R} ^{}
)$ vanishes. So $\omega = e ^{g} \omega '$ with $\omega '$
symplectic. As the lift $\widetilde{v} $ of $v$ is
$\widetilde{J} $-holomorphic we have 
$\int _{\Sigma'}
\widetilde{v}  ^{*} \widetilde{\omega }  >0$, which 
implies  $\int _{\Sigma'} \widetilde{v}  ^{*}
\widetilde{\omega}'  >0,$
i.e. $ \langle [\widetilde{v} ], \widetilde{\omega }' \rangle
> 0 $.
Now by assumption also $\langle [\widetilde{u}],
\widetilde{\omega }' \rangle > 0  $ and the conclusion
readily follows.
\end{proof}

% the completion (pre-compactification) by Kontsevich stable maps of the moduli space of pairs $(u,t)$ $(u,t)$, $u$ is a class $A$, $J _{t} $-holomorphic, genus $0$ curve in $M$ with $n$-marked points.
\section{Main argument}
\label{section:mainargument} We will first
state and prove a more general result, from which
Theorem \ref{cor:nonsqueezingintro} will be
deduced. 

Let $M=S ^{2} \times T ^{2n-2}$. We have  real
codimension 1 hypersurfaces $$\Sigma _{i} =S ^{2}
\times  (S ^{1} \times \ldots \times S ^{1}
\times   \{pt \} \times S ^{1}   \times \ldots
\times S ^{1})   \subset M,   $$ where the
singleton $\{pt\} \subset S ^{1}  $ replaces the
$i$'th factor of $T ^{2n-2}= S ^{1} \times \ldots
\times S ^{1}  $. The hypersurfaces $\Sigma _{i} $
are naturally foliated by the symplectic submanifolds 
 $$M _{\theta} = S ^{2}  \times  (S ^{1} \times \ldots \times S ^{1}  \times \{pt\} \times \{\theta \} \times S ^{1}   \times \ldots \times S ^{1})  \simeq S ^{2} \times T ^{2n-2},$$ $\theta \in S ^{1} $. We denote by $T ^{fol} \Sigma _{i} \subset TM  $, the distribution of vectors tangent to the leaves of the above-mentioned foliation. In other words 
$$T ^{fol} \Sigma _{i} = \cup _{\theta} i _{*} TM
   _{\theta},     $$ where $i: M _{\theta} \to M $
   are the inclusion maps. 
Set $\Sigma = \bigcup _{i} \Sigma _{i} $, and $U = M - \Sigma $.  
\begin{theorem} \label{cor:nonsqueezing} Let $A,\omega$ and
$M =S ^{2} \times T ^{2n-2}  $, be as before s.t. 
 $ \langle \omega, A\rangle = \pi r ^{2} $. Let $\{\omega _{t} \} $, $t \in [0,1]$, $\omega _{0} = \omega$ be a $\mathcal{T}^{0}$ continuous family of $\lcs$ forms on $M$. Set $R>r$, then there is an $\epsilon >0$  s.t. if $d _{0}  (\omega _{t}, \omega _{0}) < \epsilon$ for all $t$, then
there is no symplectic embedding $$\phi: (B _{R}, \omega
_{st}) \hookrightarrow U, $$ meaning an embedding $\phi$ such that $\phi ^{*} \omega _{1} = \omega _{st}   $.

More generally, there is no symplectic embedding $$\phi: (B _{R}, \omega _{st})  \hookrightarrow (M, \omega _{1}), $$  s.t 
   $\phi _{*} j$   \text{ preserves the bundle }
   $T ^{fol} \Sigma _{i},$ for $j$ the standard
   almost complex structure on $B _{R}$, whenever
   $\phi (x) \in \Sigma _{i}$. In other words, 
\begin{equation} \label{eq:complexlinear}
\phi _{*}j (T ^{fol} \Sigma _{i})
   \subset T ^{fol} \Sigma _{i} \subset TM,  
\end{equation}
whenever $\phi (x) \in \Sigma _{i}$.    
\end{theorem}
Let us elaborate a bit. Assuming there is no volume
obstruction, (and of course this can be arranged)  then of
course there is a volume preserving counterexample $\phi$
to the theorem. Moreover, given a symplectic
counterexample $\phi$, which necessarily does not satisfy the condition \eqref{eq:complexlinear}, it
should be possible to deform it to a symplectic embedding which
does satisfy this condition. This of course would be
a contradiction to the theorem, and so this indicates that
the condition \eqref{eq:complexlinear} might be removable. 
\begin{proof} [Proof of Theorem \ref{cor:nonsqueezing}] \label{section:proofnonsqueezing}
The second part of the theorem vacuously implies
the first, and we proceed with the proof of the
second part. Fix an $\epsilon' >0$ s.t. any 2-form
$\omega _{1}  $ on $M$, $C ^{0} $
$\epsilon'$-close to $\omega$, is
non-degenerate and is non-degenerate on the
leaves of the foliation of each $\Sigma _{i}
$, discussed prior to the formulation of the
theorem. Suppose by contradiction that for
every $\epsilon>0$ there is a $\mathcal{T}
^{0}$ continuous homotopy $\{\omega _{t} \}$ of $\lcs$ forms, with $\omega _{0}=\omega$, such that $\forall t: d _{0} (\omega_{t}, \omega) < \epsilon$ and such that 
there exists a symplectic embedding $$\phi: B _{R}  \hookrightarrow (M, \omega _{1}), $$   s.t 
$$\phi _{*}j (T ^{fol} \Sigma _{i})
\subset T ^{fol} \Sigma _{i} \subset TM,  $$
whenever $\phi (x) \in \Sigma _{i}$.

Take $\epsilon<\epsilon'$, and let $\{\omega _{t} \}$ be as
in the hypothesis above. In particular, $\omega _{t} $ is an
$\lcs$ form for each $t$, and is non-degenerate on the
leaves of $\Sigma _{i} $.
Extend $\phi _{*}j $ to an $\omega _{1} $-compatible
almost complex structure $J _{1} $ on $M$, preserving $T
^{fol}  \Sigma _{i} $ for each $i$. We may then extend
this to a family $\{J _{t} \}  $ of almost complex
structures on $M$, s.t. $J _{t} $ is $\omega _{t}
$-compatible for each $t$, with $J _{0} $ is the
standard split complex structure on $M$ and such that $J
_{t} $ preserves $T ^{fol} \Sigma _{i} $ for each $t,i$.
The latter condition can be satisfied since the leaves of $\Sigma _{i}
$ are $\omega _{t} $-symplectic for each $i,t$.  When
$\phi (B _{R})$ does not intersect $\Sigma $ these
conditions can be trivially satisfied. First find an
extension $J _{1} $ of $\phi _{*} j $ preserving $T
^{fol} \Sigma _{i} $ for each $i$. Then extend $J _{1}$ to a family $\{J _{t} \}$.

Now the family $\{(\omega _{t}, J _{t}  )\}$
satisfies the hypothesis of Proposition
\ref{thm:noSkycatastrophe} for the class $A =
[S ^{2} ] \otimes [pt]$ as in the statement of
the theorem we are proving. Then by Proposition
\ref{thm:noSkycatastrophe} $L ^{2}$ energy $e$ is bounded on   
$$C=\overline{\mathcal{M}} _{0,1}   (\{J _{t} \},
   A)$$ and hence $C$ is compact by Theorem \ref{thm:complete}. 

The classical Gromov-Witten invariant counting class $A$, $J
_{0} $-holomorphic, genus 0 curves passing through a fixed
point is:
\begin{equation*}
GW _{0,1} (A,J _{0} )([pt]) =1,
\end{equation*}
whose calculation already appears in
\cite{citeGromovPseudoholomorphiccurvesinsymplecticmanifolds.}.
Then by compactness of $C$, and the discussion preceding the proof: $$GW _{0,1} (A,J _{1} ) ([pt]) =1.$$ 
In particular there is a class $A$ $J
   _{1}$-holomorphic curve $u: \mathbb{CP} ^{1}
   \to M$ passing through $\phi ({0}) $.  

   By Lemma \ref{lemma:K} we may choose a lift $\widetilde{u} $ of $u$ to $\widetilde{M} $, with homology class $[\widetilde{u} ]$ also denoted by $A$ so that the image of $\widetilde{u} $ is contained in a compact set $K \subset \widetilde{M} $, (independent of the choice of $\{\omega _{t} \}, \{J _{t} \}$ satisfying above conditions). 
Let $\widetilde{\omega} ^{symp} _{t} $ and $f _{t} $ be as in Lemma \ref{lemma:lcsconvergence},  then by this lemma for every $\delta > 0$ we may find an $\epsilon>0$
so that if $d _{0} (\omega _{1}, \omega)< \epsilon$ then $d _{C ^{0}} (\widetilde{\omega} ^{symp},  \widetilde{\omega} _{1} ^{symp}) <\delta$ on $K$, and $\sup _{K} |f _{1} - 1|  < \delta$.

Let $\delta$ as above be chosen, and let
   $\epsilon$ correspond to this $\delta$. Now  we have: $$  |\langle \widetilde{\omega} _{1}    ^{symp}, A  \rangle - \pi \cdot r ^{2}| = 
|\langle \widetilde{\omega} _{1}  ^{symp}, A  \rangle - 
\langle \widetilde{\omega}  ^{symp}, A  \rangle |
 \leq \delta \pi \cdot r ^{2},$$
 as $ \langle \widetilde{\omega} ^{symp}, A
   \rangle =\pi r ^{2}  $, and as $d _{C ^{0}}
   (\widetilde{\omega} ^{symp},
   \widetilde{\omega} _{1} ^{symp}) <\delta$.  
And we have $$\max _{K} f _{1} \leq 1+ \delta.$$ 
So choosing $\epsilon, \delta$ appropriately  we get
$$|\int _{\mathbb{CP} ^{1} } u ^{*} \omega _{1} - \pi r ^{2}    | \leq  |\max _{K} f _{1} \langle \widetilde{\omega} _{1}  ^{symp}, A  \rangle - \pi \cdot r ^{2}  | < \pi R ^{2} - \pi r^{2}.   $$   

  Consequently, $$\int _{\mathbb{CP} ^{1} } u ^{*}
\omega _{1} < \pi R ^{2}. $$

  We may then proceed exactly as in the now classical proof
	of
	Gromov~\cite{citeGromovPseudoholomorphiccurvesinsymplecticmanifolds.}
	of the non-squeezing theorem to get a contradiction and
	finish the proof.  A bit more specifically, $\phi ^{-1}
	({\image \phi \cap \image u})  $ is a  minimal surface in
	$B _{R}  $, with boundary on the boundary of $B _{R} $,
	and passing through $0 \in B _{R} $. By construction it has area strictly less then $\pi R ^{2} $ which is impossible by the classical monotonicity theorem of differential geometry.
\end{proof}

\begin{proof} [Proof of Theorem \ref{cor:nonsqueezingintro}] 
Set $U = M - \bigcup _{i} \Sigma _{i}$. Let
   $\epsilon$ be as given by the    Theorem
   \ref{cor:nonsqueezing}.  By Proposition \ref{prop:C1deformation} there is a $\epsilon'$ s.t.
   whenever $\omega _{0}, \omega _{1} \in LCS (M)
   $ are  $C ^{1}$ $\epsilon'$-close, they are
   $\mathcal{T} _{0}$  
   $\epsilon$-close.      
      
   Let $\{\omega _{t}\}$
   be given as in the hypothesis, and such that $d
   _{C ^{1}} (\omega _{0}, \omega _{t}) <
   \epsilon' $ for all $t$. 
By Proposition \ref{prop:C1deformation}.
   $\{\omega _{t}\}$ is $\mathcal{T} ^{0}$ 
   continuous,   and by the discussion above 
   $$\forall t:  d
   _{0} (\omega _{0}, \omega _{t}) <
   \epsilon. $$
   So applying Theorem
   \ref{cor:nonsqueezing} we obtain that there is
   no symplectic embedding $B _{R} \hookrightarrow
   (U, \omega _{1})$.  And so we are done.
\end{proof}
\begin{proof} [Proof of Theorem \ref{thm:Gromov}] 
	We only sketch the proof, as it is basically just a special case of Theorem \ref{cor:nonsqueezing}. For $\epsilon$ taken to be sufficiently small, the family
	$\omega _{t} = t \omega _{0} + (1-t) \omega'  $ is
	a family of symplectic forms on $M$. Then proceed as in
	the proof Theorem \ref{cor:nonsqueezing}, upon noting that we
	do not need additional assumptions on the embedding
	$\phi$ or the family $\omega _{t}$, to have compactness of
	the relevant moduli spaces. So compactness is automatic,
	and the proof goes through as before.
\end{proof}

\section{Acknowledgements} 
I am grateful to  Kevin Sackel, Richard Hind and Vestislav
Apostolov for related discussions, as well as the referee
for nice suggestions.
\bibliographystyle{siam}  
%  \bibliography{/root/texmf/bibtex/bib/link}  
 % \bibliography{/home/yashasavelyev/texmf/bibtex/bib/link} 
\bibliography{/Users/yasha/texmf/bibtex/bib/link} 
% \bibliography{/home/yasha/texmf/bibtex/bib/link} 
\end {document}